\documentclass{amsart}
\usepackage{amssymb, amsmath, latexsym, a4,harmony}
\usepackage[latin1]{inputenc}
\usepackage[2cell,arrow,curve,matrix]{xy}
\usepackage{a4wide, amssymb, bbm, calligra, mathrsfs}
\UseHalfTwocells \UseTwocells

\def\xto#1{\xrightarrow[]{#1}}

\usepackage{enumerate}
\usepackage{MnSymbol}
\usepackage{color}
\usepackage{setspace}
\usepackage{fancyhdr}

\def\xto#1{\xrightarrow[]{#1}}

\def\id{\sf Id}

\def \ker{\mathop{\sf Ker}\nolimits}
\def \cok{\mathop{\sf Coker}\nolimits}

\def \Hom{\mathop{\sf Hom}\nolimits}
\def \Pic{\mathop{\sf Pic}\nolimits}

\def\1{^{-1}}
\newtheorem{De}{Definition}[section]
\newtheorem{Th}[De]{Theorem}
\newtheorem{Pro}[De]{Proposition}
\newtheorem{Le}[De]{Lemma}
\newtheorem{Co}[De]{Corollary}

\newtheorem{Rem}[De]{Remark}

\def\Pic{\mathfrak{Pic}}
\def \K{\mathfrak{K}}

\def \G{\mathfrak{G}}
\def \H{\mathfrak{H}}

\def \lim {\mathop{\sf lim}\nolimits}

\def \0{\mathsf{0}}

\def \tqu_f{\overline{\mathfrak{Qu}}_f}
\def \QU_f{\mathfrak{QU}_f}
\def \Qu_f{\mathfrak{Qu}_f}
\def \Q_f{\mathfrak{Q}_f} 
 
\def \QX_f{\mathfrak{Q}_{X,f}} 
\def \QU{\mathfrak{QU}}
\def \Qu{\mathfrak{Qu}}
\def \Q{\mathfrak{Q}}
\def \sQU_f{\mathsf{QU_f}}
\def \sQu_f{\mathsf{Qu_f}}
\def \sQ_f{\mathsf{Q_f}}
\def \sQX_f{\mathsf{Q}_{X,f}} 
\def \sQU{\mathsf{QU}}
\def \sQu{\mathsf{Qu}}
\def \sQ{\mathsf{Q}}

\def \sU{\mathsf{U}}
\def \dis{\mathsf{Dis}}

\def \Dis{\mathfrak{Dis}}

\begin{document}
\title{On the group of separable quadratic algebras and stacks}

\author[I. Pirashvili]{Ilia Pirashvili}
\email{ilia\_p@ymail.com}

\maketitle

\begin{abstract} 

The aim of this paper is to study the group of isomorphism classes of torsors of finite flat group schemes of rank 2 over a commutative ring $R$. This, in particular, generalises the group of quadratic algebras (free or projective), which is especially well studied. Our approach however, yields new results even in this case.

\end{abstract}

\section{Introduction}

We are going to study the group of separable quadratic algebras, which was introduced by Bass in \cite[Section 4.3] {bass_tata} and was extensively studied by several authors (see for example \cite{small}, \cite{hahn} and references therein), using the methods of 2-dimensional categorical algebra. Our methods are based on the fact that stackification is an exact 2-functor.

Let $R$ be a commutative ring and we denote by $\sQu(R)$ the group of isomorphism classes of separable quadratic $R$-algebras \cite[Ch. 12]{hahn} and by ${\sf Dis}(R)$ the group of isomorphism classes of nonsingular symmetric bilinear $R$-modules, which are finitely generated and projective of rank $1$ \cite[Ch. 4C]{hahn}. According to \cite[Statement 12.6, p.178]{hahn}, one has an exact sequence of abelian groups
$$\sQu(R)\to {\sf Dis}(R)\to {\sf Dis}(R/4R).$$  
Our goal is to extend this important exact sequence on the left hand side.  It follows from our results that  the following  is true:

\begin{Th}\label{exact} If $2$ is not a zero-divisor in $R$, one has an exact sequence of abelian groups 
$$0\to \mathbb{Z}_2 (R)\to \mu_2(R)\to {\sf Im}(\mu_2(R/4R)\to \mu_2(R/2R))\to \sQu(R)\to {\sf Dis}(R)\to {\sf Dis}(R/4R).$$
Here 
$$\mu_2(R)=\{r\in R| r^2=1\} \text{ and } \mathbb{Z}_2 (R)=\{r\in R| r^2=r\}.$$ 

The operation in $\mathbb{Z}_2$ is given by $r+_1 s=r+s-2rs$ (see \cite[p. 202]{hahn}). The first nontrivial map  is given by $r\mapsto 1-2r$.
\end{Th}  

In order to obtain the above, we are going to define a symmetric categorical group $\Qu(R)$ such that the groups $\sQu(R)$ and $\mathbb{Z}_2(R)$ are its $\pi_0$ and $\pi_1$ respectively (see Section \ref{cgr} for symmetric categorical groups). Similarly, behind of the groups ${\sf Dis}(R)$ and $\mu_2(R)$ there is a symmetric categorical group $\Dis(R)$, while the groups ${\sf Dis}(R/4R)$ and ${\sf Im}(\mu_2(R/4R)\to \mu_2(R/2R))$ are obtained from the symmetric  categorical group $\Dis^{-2,1}(R)$. Hence, it suffices to show that $\Qu(R)$, $\Dis(R)$ and $\Dis^{-2,1}(R)$ fit in an exact sequence of symmetric categorical groups, and then apply the classical six-term exact sequence \cite[p.84] {gz} involving $\pi_i$, $i=0,1$. The symmetric categorical groups $\Qu(R)$ and $\Dis(R)$ are classical and they  appear (at least implicitly) in many papers, for example in \cite{hahn}. However, $\Dis^{-2,1}(R)$  is new. The next observation is that, by varying $R$, these symmetric categorical groups form stacks in the Zariski topology. 

As it turns out, they can be seen as the stackifications of the analogue prestacks of symmetric categorical groups, based on free modules instated of projective ones. It is a fairly easy exercise to show that the 'free' versions of the above symmetric categorical groups fit in an exact sequence. Since stackification sends exact sequences to exact ones, the result then follows.   
\newline

Actually, we will work in a more general setting than that of separable quadratic algebras. In fact, we will fix a commutative and cocommutative Hopf algebra $J$, which is free of rank $2$ as an $R$-module. It is a classical fact that any such Hopf algebra is uniquely determined by a pair $(p,q)$ of elements of $R$, such that $pq+2=0$ (see \cite[Theorem 1.2]{kr}). Our main object of interest will be the group $\sQu^{pq}(R)$ of $J$-Galois algebras, with our main result being Theorem \ref{exactpq}. Since $\sQu(R)=\sQu^{-2,1}(R)$,  Theorem \ref{exact} is a specialization of Theorem \ref{exactpq}.
    
\section{Cat-groups} \label{cgr}
Recall that a \emph{symmetric categorical group}, or \emph{cat-group} for short, is a groupoid $\G$, equipped with a symmetric monoidal structure $+:\G\times\G\to \G$, such that for each object $X$, the endofunctor $(-)+X:\G\to \G$ is an equivalence of categories \cite{v}. In particular, one has an object $Y$ and an isomorphism $X+Y\to 0$, where $0$ is the neutral object with respect to the symmetric monoidal structure $+$.  
If $\G$ and $\H$ are cat-groups, a morphism $\alpha:\G\to \H$ is just  a symmetric monoidal functor. 

Let  $\G$ be a cat-group. The monoidal structure induces an abelian group structure on the set of connected components. This abelian group is denoted by $\pi_0(\G)$. We also have an abelian group $\pi_1(\G)$, which is the group of automorphisms of the neutral object of $\G$. Since the functor $(-)+X:\G\to \G$ is an equivalence of categories for every object $X$, the  induced homomorphism $\pi_1(\G)\to Aut_\G(X)$ is an isomorphism. Because of this, we will identify the automorphism group of any object with $\pi_1(\G)$. 
Recall also that if $\alpha:\G\to \H$ is a morphism of cat-groups, than $\alpha$ is an equivalence if and only if the induced morphisms $$\pi_i(\G)\to\pi_i(\H), \ \ i=0,1$$
are isomorphisms.

We also note that if $\alpha:\G\to \H$ is a morphism of cat-groups, than one can define the cat-group $2\text{-}\ker(\alpha)$. 
Objects of $2\text{-}\ker(\alpha)$ are pairs $(X,x)$, where $X$ is an object of $\G$, while $x:0\to\alpha(X)$ is a morphism in $\H$. A morphism $(X,x)\to (Y,y)$ in $2\text{-}\ker(\alpha)$ is a morphism $f:X\to Y$ in $\G$, such that $y=\alpha(f)x$. 
It is also well-known (\cite[p. 84] {gz}, \cite{v}) that there is an exact sequences of abelian groups
\begin{equation}\label{ex_gz}
0\to\pi_1(2\text{-}\ker(\alpha))\to \pi_1(\G)\to \pi_1(\H)\to \pi_0(2\text{-}\ker(\alpha))\to \pi_0(\G)\to \pi_0(\H).
\end{equation}
Recall the definition of an exact sequence of cat-groups \cite{v}. A diagram
$$
\xymatrix{ \G\ar[r]_{\alpha}\rruppertwocell<12>^{0}{^\kappa}
&\H\ar[r]_{\beta}&\K}
$$
is called \emph{exact} at $\H$, provided the induced functor $\gamma:\G\to 2\text{-}\ker(\beta)$ is essentially surjective. More precisely, 
$$\gamma(X)=(\alpha(X),\kappa(X): \beta\alpha (X)\to 0).$$ 
It follows that one has an exact sequence of abelian groups
$$\pi_0(\G)\to \pi_0(\H)\to \pi_0(\K).$$

If we are given a morphism $\alpha:G\rightarrow H$ of abelian groups, than we can make a cat-group $\G_{\alpha}$ whose objects are the elements of $H$, and for $h_1,h_2\in H$, we have 
$$\Hom(h_1,h_2)=\{g\in G|h_2+\alpha(g)=h_1\}.$$ 
The monoidal structure is induced by the additive group structure on $G$ and $H$. We have $\pi_0(\G_\alpha)=\cok(\alpha)$ and $\pi_1(\G_\alpha)=\ker(\alpha)$. 
 
 Let $\alpha:G_0\rightarrow G_1$ and $\beta:H_0\rightarrow H_1$ be morphism of abelian groups and $\G_\alpha$ and $\G_\beta$ their associated cat-groups. Observe that any  commutative diagram  of abelian groups and group homomorphisms
$$\xymatrix{G_0\ar[r]\ar[d]_\alpha & H_0\ar[d]^\beta \\ 
			G_1\ar[r] & H_1 }$$
gives rise to  a morphism of the abelian cat-groups $\G_\alpha\rightarrow \G_\beta$. We denote by $\mathsf{Pb}$ the pullback of the diagram 
$$\xymatrix{ & H_0\ar[d]^\beta \\ 
	G_1\ar[r] & H_1 }$$
The following result is well known:
\begin{Pro}\label{2ker} Let $\G_\alpha$ and $\G_\beta$ be as above. We have an equivalence of cat-groups 
$$2\text{-}\ker(\G_\alpha\rightarrow \G_\beta)\cong \G_\text{\i},$$
where $\G_\text{\i}$ is the abelian cat-group associated to the morphism \i\ $:G_0\rightarrow \mathsf{Pb}$.
\end{Pro}
\begin{Co}\label{liftedker}
	Let $\alpha:G\rightarrow H$ be a homomorphism of abelian groups and denote the associated cat-groups of $G\rightarrow 0$ and $H\rightarrow 0$ by $\G_G$ and $\G_H$ respectively. There is an equivalence of cat-groups
	$$2\text{-}\ker(\G_G\rightarrow \G_H)\xrightarrow[\sim]{\alpha_*}\G_\alpha.$$
\end{Co}

\section{Free J-Galois Algebras}

\subsection{The cat-group $\Qu_f^{pq}(R)$} Let $R$ be a commutative ring and $p$ and $q$ two elements of $R$ such that $pq+2=0$. These elements will be fixed during this section. Of special interest is the case $p=-2$, $q=1$.

Objects of the groupoid $\Qu^{pq}_f(R)$ are pairs  
$[a,b]$, where $a,b\in R$ and $a^2+p^2b\in R^*$. Morphisms $[c,d]\to [a,b]$ are pairs $(u,r)$,  such that $r\in R, u\in R^*$, and the following relations hold:
\begin{equation}\label{41}
c=au-pr \ \ \ \ {\rm and} \ \ \ d=u^2b-qrua-r^2.
\end{equation}
The composition of morphisms is defined by
$$(v,s)\circ (u,r)=(vu, rv+s).$$ 
The morphism $(1,0):[a,b]\to [a,b]$ is the identity morphism of $[a,b]$.
We define the functor 
$$\star:\Qu_f^{pq}(R)\times \Qu_f^{pq}(R)\to \Qu_f^{pq}(R)$$ 
as follows: On objects we have
$$[a_1,b_1]\star[a_2,b_2]=[a_1a_2,a_1^2b_2+a_2^2b_1+p^2b_1b_2],$$
where we write $[a_1,b_1]\star[a_2,b_2]$ instead of $\star([a_1,b_1],[a_2,b_2])$, and on morphisms $\star$ is given by 
$$(u_1,r_1)\star (u_2,r_2)=(u_1u_2,-pr_1r_2+r_1u_2a_2+u_1a_1r_2).$$
One checks that $\star$ defines a symmetric monoidal structure on $\Qu_f^{pq}(R)$. The neutral object is $[1,0]$.
Since the pair $(a^2+p^2b,pb)$ defines a morphism $ [a,b]\star[a,b]\to [1,0]$, the symmetric monoidal groupoid $\Qu_f^{pq}(R)$ is in fact a cat-group.
Consider the set
$$\mathbb{Z}_{pq}(R)=\{r\in R | r^2=qr\},$$
which is considered as an abelian group via the operation $r+_1 s=r+s+prs$. It is clear that the map $r\mapsto (1+pr,r)$ yields an isomorphism of groups 
$$\mathbb{Z}_{pq}(R)\xto{\simeq} \pi_1(\Qu_f^{pq}(R)).$$ 
We let $\sQu^{pq}_f(R)$ be the group of isomorphism classes of $\Qu^{pq}_f(R)$, that is $\sQu^{pq}_f(R)=\pi_0(\Qu^{pq}_f(R))$. 
 
Observe that if $p=-2$ and $q=1$, the above relations are exactly the same as in \cite[pp. 29-30]{hahn}. Hence, $\sQu^{-2,1}_f(R)$ is the same as the \emph{free quadratic  group} introduced in \cite[p. 31]{hahn}, where it is denoted by $\sQu_f(R)$. Recall that the group $\sQu_f(R)$ is the group of isomorphism classes of free separable $R$-algebras of rank two \cite[Ch. 3A] {hahn}. We will see in Proposition \ref{43} that the groups $\sQu^{pq}(R)$ have a similar interpretation for all $(p,q)$. See \cite{hahn} for an extensive study of $\sQu_f(R)$. The group $\mathbb{Z}_{-2,1}(R)$ is the group of idempotents of $R$ and it appears in many places, see for example \cite[p. 202]{hahn}, where it is denoted by $\mathbb{Z}_2(R)$.
 
\begin{Pro}[An equivalence of categories]\label{eqcat} 
Let $(p,q)$ be a pair of elements in $R$ such that $pq+2=0$. For any invertible element $t\in R^*$, there is an equivalence of categories
$$t_*:\Qu_f^{pq}(R)\to \Qu_f^{pt,t^{-1}q}(R).$$
\end{Pro}

On objects, the functor $t_*$ is given by
$$t_*([a,b])=[at,b],$$
while on morphism it is given by $t_*(u,r)=(u,r)$.

\subsection{The Cat-group $\G(R)$ and the morphism $\alpha_f:\Qu_f^{pq}(R)\to \G(R)$}

Consider  the homomorphism of abelian groups   
$$R^*\xto{sq} R^*, \ \ sq(r)= r^2.$$ 
Denote its associated cat-group by $\G(R)$. We have $$\pi_0(\G(R))=\sU_2(R) \ \ \ {\rm and} \ \ \ \ \ \pi_1(\G(R))=\mu_2(R),$$
where $$\sU_2(R)=R^*/(R^*)^2 \ \ \ \ \ \ {\rm and} \ \ \ \mu_2(R)=\{r\in R| r^2=1\}.$$
Define the morphism of categorical groups
$$\alpha_f:\Qu_f^{pq}(R)\to \G(R)$$
as follows: For an object $[a,b]$ of $\Qu_f^{pq}(R)$ one puts
$$\alpha_f([a,b])=a^2+p^2b.$$
If  $(u,r): [c,d]\to [a,b]$  is a morphism in $\Qu_f(R)$, then we have 
$$c^2+p^2d=(au-pr)^2+p^2(u^2b-qrua-r^2)=u^2(a^2+p^2b)$$ 
thanks to  equalities in (\ref{41}). Thus, $u$ can be  considered as a morphism $ c^2+4d \to a^2+4b$ in $\G(R)$ and we put
$$\alpha_f(u,r)=u: \alpha([c,d])\to \alpha([a,b]).$$ 

The morphism of cat-groups $\alpha_f:\Qu_f^{pq}(R)\to \G(R)$ induces the homomorphisms of abelian groups
$$\mathbb{Z}_{pq}(R)=\pi_1(\Qu_f(R))\to\pi_1(\G(R))=\mu_2(R)$$
and
$$\sQu_f^{pq}(R)\to \sU_2(R).$$
The first homomorphism is given by $a\mapsto 1+pa$. It is well known that for $p=-2, q=1$ the second homomorphism fits in the exact sequence (see \cite[(3.6) on p.32]{hahn})
\begin{equation}\label{2}\sQu_f(R)\to \sU_2(R)\to \sU_2(R/4R).\end{equation}
We wish to extend this sequence on the left hand side (for arbitrary $(p,q)$), as well as lift it to the level of cat-groups. This will yield analogous exact sequences after stackification. Observe that the obvious candidate
$$\Qu_f^{pq}(R)\to \G(R)\to \G(R/p^2R)$$
does not work, as the composite functor $\Qu_f^{pq}(R)\to \G(R/p^2R)$ is not the trivial one, not even when $p=-2, q=1$. As such, we have to modify the last cat-group. 

\subsection{The Cat-group $\G^{pq}(R)$ and the morphism $\beta_f:\G(R)\to \G^{pq}(R)$} 

We will need the following easy and well-known fact.
\begin{Le} Let $R$ be a commutative ring and $r\in (R/pR)^*$. If $\tilde{r}\in R/p^2R$ is a lifting of $r$, then $\tilde{r}\in (R/p^2R)^*$. Moreover, there is a well-defined homomorphism  $$Sq:(R/pR)^*\to (R/p^2R)^*$$ given by $Sq(r)=\tilde{r}^2$. 
\end{Le}
\begin{proof} If $r$ is invertible and $s=r^{-1}$ in $R/pR$, then $\tilde{r}\tilde{s}=1+px$ for some $x\in R/p^2R$. Here $\tilde{s}$ is a lifting of $s$ in $R/p^2R$ and we obtain $\tilde{r}\tilde{s}(1-px)=1$. Hence, $\tilde{r}$ is invertible. To show the last assertion, observe that for any $x\in R/p^2R$, one has $(\tilde{r}+px)^2=\tilde{r}^2+2prx=\tilde{r}^2$. This is because $2=-pq$ and $p^2=0$ in $R/p^2$. 
\end{proof}

Let  $\G^{pq}(R)$ be  the cat-group corresponding to the homomorphism $Sq:(R/pR)^*\to (R/p^2R)^*.$ \begin{Le}\label{22}
We have $$\pi_0(\G^{pq}(R))=\sU_2(R/p^2R)$$  and
$$ \pi_1(\G^{pq}(R))=\ker( (R/pR)^*\xto{Sq} (R/p^2R)^*) \cong {\sf Im}(\mu_2(R/p^2R)\to \mu_2(R/pR)).$$
\end{Le}
\begin{proof} To see the first isomorphism, it suffices to note that ${\sf Im}(Sq)={\sf Im}((R/p^2R)^*\xto{sq} (R/p^2R)^*)$. The second statement is obvious.
\end{proof}

Since the following diagram of abelian groups
$$\xymatrix{R^*\ar[r]\ar[d]_{sq}& (R/pR)^*\ar[d]^{Sq}\\ R^*\ar[r]&(R/p^2R)^*}$$
commutes, it gives rise to a morphism of  categorical groups $$\beta_f:\G(R)\to \G^{pq}(R).$$

\subsection{The natural transformation  $\delta_f:\beta_f\alpha_f\to 1$} 

Take an object $[a,b]$ in $\Qu_f(R)$. Since $a^2+p^2b\in R^*$, the element $a(mod \ pR)$ is invertible in $R/pR$. As such, it can be considered as a morphism $a^2 ( mod \ p^2R)\to 1$ in $\G^{pq}(R)$. We denote this map by $\delta_f([a,b])$ and obtain $\delta_f([a,b]):(\beta_f\circ \alpha_f)([a,b])\to 1$. Moreover, if $(u,r):[c,d]\to [a,b]$ is a morphism of $\Qu_f^{pq}(R)$, one has $c=ua (mod \ pR)$, thanks to the equations in (\ref{41}) on page \pageref{41}. This means that the following diagram 
$$\xymatrix{c^2(  mod \ p^2R) \ar[rr]^{\delta_f([c,d])} \ar[d]_{u (mod \ pR)}&&1\\
 a^2(mod \ p^2R)\ar[rru]_{\delta_f([a,b]) }&&
}$$
commutes in  $\G^{pq}(R)$.  Thus, $\delta_f: \beta_f\alpha_f\to 1$ is a natural transformation from $\beta_f\circ \alpha_f$  to the trivial functor.
\subsection{$\Qu_f^{pq}(R)$ as a 2-kernel} We can now formulate the main result of this section.
\begin{Th} \label{exact_{fpq}} i) The  sequence of cat-groups
 $$
\xymatrix{ \Qu_f^{pq}(R)\ar[r]_{\alpha_f}\rruppertwocell<12>^{1}{^\delta_f\,}
&\G(R)\ar[r]_{\beta_f}&\G^{pq}(R)
}$$
is exact at $\G(R)$,  that is the induced functor $$\gamma_f:\Qu_f^{pq}(R)\to 2\text{-}\ker(\beta_f)$$ is essentially surjective.

ii) If $p$ is not a zero-divisor in $R$, then $\gamma_f$ is an equivalence of cat-groups.
\end{Th}

\begin{proof} i) According to Proposition \ref{2ker}, the cat-group $2\text{-}\ker(\G(R)\rightarrow \G^{pq}(R))$ is equivalent to the cat-group $\G_h$ associated to the following homomorphism of abelian groups: 
$$R^*\xto{h} \{(x,s)| x\in R^*, s\in (R/pR)^*, x\equiv \tilde{s}^2 (mod \ p^2R)\}.$$
Here $\tilde{s}\in R/p^2R$ is a lifting of $s$ and $$h(u)=(u^2, u \,(mod \ pR)).$$
Hence, the induced functor
$$\gamma:\Qu_f(R)\to \G_h$$
is given by
$$ \gamma_f([a,b])=(a^2+p^2b, a \,(mod \ pR)) \ \ \ {\rm (on \ objects)}$$
$$\gamma_f(u,r)=u \ \ \ {\rm (on \ morphisms)}.$$
Take an object $(x,s)$ in $\G_h$. By assumption $x\in R^*$ and  there are elements $a,b\in R$ such that $x=a^2+p^2b$ and $s=a\, (mod \ pR)$. Thus $\gamma_f([a,b])=(x,s)$, and so $\gamma_f$ is surjective on objects. 

ii)  Thanks to part i), we only need to show that the functor $\gamma_f$ is full and faithful. Take two objects $[a,b]$ and $[c,d]$ in $\Qu_f^{pq}(R)$ and let $(u,r)$ and $(u',r')$ be two morphisms $[c,d]\to [a,b]$. In particular, $c=ua-pr=u'a-pr'$. If $\gamma_f(u,r)=\gamma_f(u',r')$ then $u=u'$, and so $pr=pr'$. Thus $r=r'$, implying that $\gamma_f$ is faithful. Take a morphism $u:\gamma_f([a,b])\to \gamma_f([c,d])$ in $\G_h$. By definition, $u\in R^*$ and the following two conditions hold:
$$c^2+p^2d=u^2(a^2+p^2b) \ \ \ {\rm and} \ \ c\equiv ua\, ({\rm mod}\, pR).$$
The last condition implies that there exist an element $r\in R$ such that $c=ua+pr$. We have
$$p^2d=-c^2+u^2(a^2+p^2b)=-(ua-pr)^2+u^2(a^2+p^2b)=p^2(u^2b-qrua-r^2).$$
Since $p$ is not a zero-divisor in $R$, we see that $d=u^2b-qrua-r^2$. This, together with the equality $c=ua+pr$, shows that the pair $(u,r)$ is a morphism $[c,d]\to [a,b]$ in $\Qu_f(R)$, implying that $\gamma_f$ is full.
\end{proof}

\begin{Co} \label{44}  i) For any commutative ring $R$, one has an exact sequence of abelian groups
$${\sf Qu}_f^{pq}(R)\to \sU_2(R)\to \sU_2(R/p^2R).$$
ii)  If $p$ is not a zero divisor in $R$, one has an exact sequence of abelian groups
$$0\to \mathbb{Z}_{pq}(R)\to \mu_2(R)\to {\sf Im}(\mu_2(R/p^2R)\to \mu_2(R/pR))\to \sQu_f(R)\to \sU_2(R)\to \sU_2(R/p^2R).$$
iii)  For any commutative ring $R$, the functor $\Qu^{1,-2}_f(R)\to \G(R)$ is an equivalence of cat-groups. In particular, ${\sf Qu}_f^{1,-2}(R)\cong \sU_2(R)$.
\end{Co}

\begin{proof} i)  By applying  $\pi_0$ to the exact sequence of cat-groups constructed in Theorem \ref{exact_{fpq}} i), one obtains the exact sequence i). Similarly, exact sequence ii) is a specialization of the above, thanks to  Theorem \ref{exact_{fpq}} ii). If $p=1$, $p$ is not a zero-divisor and as such we can use Theorem \ref{exact_{fpq}} ii). The result now follows since $\G^{1,-2}(R)$ is the trivial cat-group. 
\end{proof}

{\bf Remarks} i) The exact sequence in Corollary \ref{44} i) in the case $p=-2, q=1$ is well-known  (see the statement \cite[(3.6) on p.32]{hahn})). The second exacts sequence however is new, even in the case $p=-2,q=1$. 
  
ii)\label{remii} If $2$ is invertible in $R$, the categories $\Qu^{-2,1}_f(R)$ and  $\Qu^{1,-2}_f(R)$ are equivalent (see Proposition \ref{eqcat}). In this case we recover the well-known isomorphism ${\sf Qu}_f^{-2,1}(R)\cong \sU_2(R)$  see \cite[Statement (3.4) i). on p.32]{hahn}.
  
\subsection{An exact sequence} 

In this section we generalize the result \cite[Exercises 6-10, p.41]{hahn} to arbitrary $\sQu^{pq}(R)$. Define the set 
$${\mathcal V}_1(R)=\{r\in R|1+pr\in R^*\}$$ 
and equip it with the operation $+_1$ by declaring
$$r+_1 s:=r+s+prs.$$
One easily sees that under this operation ${\mathcal V}_1(R)$ is an abelian group ($0$ is also the zero with respect to $+_1$). We also need the following set 
$${\mathcal V}_2(R)=\{x\in R| 1+p^2x\in R^* \},$$ 
which is also an abelian group under the operation $+_2$, where 
$$x+_2 y:=x+y+p^2xy.$$ 
It is straightforward to check that the map
$$\zeta:{\mathcal V}_1(R)\to {\mathcal V}_2(R), \ \ \  \zeta(r)=r^2-qr$$
is a group homomorphism and as such defines a cat-group, which we denote by ${\frak V}^{pq}(R)$. We have
$$\pi_1({\frak V}^{pq}(R))=\{r\in R| r^2=qr \ {\rm and} \ 1+pr\in R^*\},$$
whose group structure agrees with $+_1$. We also set
$${\sf V}^{pq}(R)=\pi_0({\frak V}^{pq}(R)).$$
In the case when $p=-2$ and $q=1$, this group is the same group as one defined in  \cite[Exercise 6, p.41]{hahn}. We construct the symmetric monoidal functor
$$\rho: {\frak V}^{pq}(R) \to \Qu_f^{pq}(R)$$
as follows: On objects it is given  by $\rho(x)=[1,x]$. If $r$ is a morphism from $y$ to $x$ in ${\frak V}^{pq}(R)$ (that is, $y= x+(-qr+r^2)(1+p^2x)$), then $\rho(r)=(1+pr,r)$, which is considered as a morphism $[1,y]\to [1,x]$ in $\Qu_f^{pq}(R)$.

We also need a cat-group ${\frak S}^{p.q}(R)$. The objects of ${\frak S}^{p.q}(R)$ are elements $a\in R$ for which there exist $b\in R$ such that $a^2+p^2b\in R^*$.  Let $a$ and $c$ be objects of ${\frak S}^{p.q}(R)$. A morphism $c\to a$ is an element $u\in R^*$, for which there exist an element $r\in R$, such that $c=au-pr$. The composition of morphisms and the symmetric monoidal structure is induced by the multiplication in $R$. One sets 
$${\sf S}^{p.q}(R):=\pi_0({\frak S}^{pq}(R)).$$

Define  a symmetric monoidal functor $\omega:\Qu^{pq}_f(R)\to {\frak S}^{pq}(R)$ as follows: On objects we set $\omega([a,b])=a$. If $(u,r):[c,d]\to [a,b]$ is a morphism of $\Qu_f ^{pq}(R)$, we set $\omega(u,r)=u$. 

As already mentioned, the following result is well-known when $p=-2$ and $q=1$ \cite[Exerecises 9 and 10, p. 41]{hahn}.
\begin{Pro} i) One has a short exact sequence of abelian groups
$$0\to {\sf V}^{pq}(R)\xto{\rho_*} \sQu_f^{pq}(R)\xto{\omega_*} {\sf S}^{pq}(R)\to 0,$$  
where $\rho_*$ and $\omega_*$ are induced by the functors  $\rho$ and  $\omega$.

ii) If $R$ is a local ring, then ${\sf S}^{pq}(R)=0$. 

iii) If $p\in Rad(R)$, then  ${\sf S}^{pq}(R)=0$. Here, $Rad(R)$ is the intersection of all maximal ideals of $R$. In particular, if $p=0$, then
$$\sQu_f^{0q}(R)=R/R_0, \ \ \  where\ \ \  R_0=\{r^2+qr| r\in R\}\subset R.$$

\end{Pro}
\begin{proof} i)  Let us show that $\omega_*$ is an epimorphism. Take an object $a$ of ${\frak S}^{pq}(R)$. By definition, there exist $b\in R$ such that $a^2+p^2b\in R^*$. Thus, $[a,b]\in \Qu_f^{pq}(R)$ and $\omega([a,b])=a$. To see that $\omega_*\circ \rho_*=0$, take an  object $x\in {\frak V}^{pq}(R)$. We have $1+p^2x\in R^*$, and hence
$$\omega\circ \rho(x)=\omega([1,x])=1.$$
Next we will show exactness at $\sQu_f^{pq}(R)$. Take an object $[a,b]$ of $\Qu_f^{pq}(R)$ such that $\omega_*([a,b])=0$. This means that there exist $u\in R^*$ and $r\in R$ such that $1=au-pr$. We set $d=u^2b-qrua-r^2$. Then $(u,r)$ is a morphism $(1,d)\to (a,b)$. Thus, the class of $[a,b]$ in $\sQ^{pq}_f(R)$ equals the class of $[1,d]=\rho(d)$ and exactness at $\sQu^{pq}_f(R)$ follows. It remains to show that  $\rho_*$ is a monomorphism. Assume $x\in R$ is an element satisfying $1+p^2x\in R^*$ for which there exists a morphism $[1,0]\to [1,x]$ in $\sQu_f^{pq}(R)$. There exist $u\in R^*$ and $r\in R$, such that $1=u-pr$ and $0=u^2x-qru-r^2$.  It follows that $u=1+pr\in R^*$ and $0=(1+pr)^2x-qr(1+pr)-r^2$. The last equation is equivalent to $0=x+(r^2-rq)(1+p^2x)$, or $0=x+_2 \zeta(r)$. So $r$ defines a morphism from $0\to x$ in ${\frak V}^{pq}(R)$. As such, the class of $x$ is zero in ${\sf V}^{pq}(R)$, proving i).

ii) Assume $R$ is a local ring with maximal ideal $\frak m$. Take an object $a$ of ${\frak S}^{pq}(R)$. We have  $a^2+p^2b=v\in R^*$ for some $b\in R$. If $p\in \frak m$, then $v^{-1}a^2=1-p^2bv^{-1}\not \in \frak m$, showing that $v^{-1}a^2$ is invertible. This implies that $a$ is invertible and $a^{-1}$ defines  a morphism $1\to a$ in ${\frak S}^{pq}(R)$. Therefore the class of $a$ is zero in ${\sf S}^{pq}(R)$. If $p\not \in \frak m$, $p$ is invertible in $R$. If $a$ is invertible, we the same argument works as before, so assume $a\in \frak m$. Then $1+a$ is invertible in $R$. Since 
$$a=1\cdot (1+a)-p\frac{1}{p},$$ 
we see that the pair $(1+a, \frac{1}{p})$ defines a morphism $a\to 1$ in ${\frak S}^{pq}(R)$. Hence, the class of $a$ is zero in ${\sf S}^{pq}(R)$, implying that ${\sf S}^{pq}(R)=0$.

iii) Take an object $a$ of ${\frak S}^{pq}(R)$, that is we have $a^2+p^2b=v\in R^*$ for some $b\in R$. Thus, $v^{-1}a^2=1-p^2bv^{-1}$ and by Nakayama $1+Rad(R)\subset R^*$. Hence, $a$ is invertible and $a^{-1}$ defines a morphism $1\to a$ in ${\frak S}^{pq}(R)$, so the class of $a$ is zero in ${\sf S}^{pq}(R)$ and ${\sf S}^{pq}(R)=0.$
\end{proof}

\section{Projective J-Galois Algebras} 

\subsection{Remarks on stacks} 

We will assume that the reader is familiar with the theory of stacks \cite{im}, which are the 2-categorical analogues of sheaves. By definition, a fibered category is nothing other than a contravariant pseudofunctor. It is additionally a stack if some gluing condition holds \cite{im}. One easily observes that if $\alpha: {\mathcal S}_1\to {\mathcal S}_2$ is a morphism of stacks, then $2\text{-}\ker(\alpha)$ is also a stack. It is well-known \cite{im} that for any  fibered category $\mathcal A$, there exists a stack ${\mathcal A}^+$ called the stackification of $\mathcal A$. This construction preserves 2-kernels. Recall also that if $\mathcal S$ is a stack of cat-groups, than $\pi_1(\mathcal S)$ is a sheaf but $\pi_0(\mathcal S)$ need not be a sheaf in general.

The underlying site which we will work with through out this paper is the Zariski site of affine schemes. However, we prefer to work with commutative rings instead of affine schemes. Accordingly, a covariant pseudofunctor from commutative rings to the 2-category of cat-groups is called a stack if the  corresponding contravariant pseudofunctor on affine schemes is one. It is clear that the assignments $R\mapsto \Qu_f^{pq}(R), \G(R), \G^{pq}(R)$ define fibered categories. Indeed, all of these are prestacks in the sense of \cite{im} because $\pi_1$ of the corresponding cat-groups are sheaves. This is clear for $\mathbb{Z}_{p.q}$ and $\mu_2$ because they are representable functors. For $\pi_1(\G^{pq}(R))$ this follows from Lemma \ref{22}. We will identify the stackifications of these prestacks and then use Theorem \ref{exact_{fpq}} to obtain an exact sequence in Theorem \ref{exactpq}. The case $p=-2$, $q=1$ yields exactly Theorem \ref{exact}, which involves the well studied (\cite{hahn}, \cite{small}) group $\sQu(R)$.
 
\subsection{The stack $\Qu^{pq}$}\label{hopf} 

Let $p$ and $q$ be two elements in $R$ such that $pq+2=0$. We let $J$  be a commutative and cocomutative Hopf algebras which is freely generated as an $R$-module by $1$ and $x$, with
$$x^2=qx  \ \ \ \rm{and}  \ \ \ \Delta(x)=x\otimes 1+1\otimes x+p \cdot x\otimes x.$$
It is well-known that any Hopf algebra, which is free of rank two as an $R$-module, is of this form \cite[Theorem 1.2]{kr}. Recall that a commutative and associative algebra with unit $A$ is called a $J$-\emph{Galois algebra}, provided $A$ is a  faithful, finitely generated, projective $R$-module, and there is given an algebra homomorphism $\eta:A\to A\otimes J$ which makes $A$ a right $J$-comodule. Further, the composition 
$$A\otimes A\xto{id\otimes \eta} A\otimes A\otimes J\xto{\mu\otimes id} A\otimes J$$ 
needs to be an isomorphism, where $\mu$ is the multiplication in $A$ \cite{kr}. The last condition means that the affine scheme $Spec(A)$ is a torsor over the group scheme $Spec(J)$.  It follows that the rank of $A$ is $2$.

Denote by $\Qu^{pq}(R)$ the cat-group of $J$-Galois algebras. Morphisms of $\Qu^{pq}(R)$ are isomorphism of $J$-Galois algebras. The monoidal structure is induced by the cotensor product of $J$. We denote the isomorphism classes of the groupoid $\Qu^{pq}(R)$ by $\sQu^{pq}(R)$ (the same group is denoted by $A(J)$ in \cite{kr}). 


 

\
Let $\mathfrak{F}:\mathfrak{A}\rightarrow \mathfrak{B}$ be a functor and consider $\{B\in\mathfrak{B} \ | \ \text{sthere are } A\in \mathfrak{A} \text{ and } \mathfrak{F}(A)\xrightarrow{\sim} B\}$. We call this the \emph{essential image} of $\mathfrak{F}$.
\begin{Pro} \label{43} i) There is a symmetric monoidal functor
$$\Qu^{pq}_f(R)\to \Qu^{pq}(R),$$ 
which sends an objects  $[a,b]$ to $(A,\eta)$. Here $A=R[t]/(t^2-aqt-b)$ while the algebra homomorphism $\eta:A\to A\otimes J$ is given by $\eta(v)=a \cdot 1\otimes x+v\otimes 1+p \cdot v\otimes x$, were $v$ denotes the class of $t$ in $A$. Moreover, if $(u,r):[c,d]\to [a,b]$ is a morphism in $\Qu^{pq}_f(R)$, the corresponding morphism $B=R[s]/(s^2-cqs-d)\to A$ is induced by $s\mapsto ut+r$.

ii) The functor $\Qu^{pq}_f(R)\to \Qu^{pq}(R)$ is full and faithful.
 
iii) The essential image of $\Qu^{pq}_f(R)\to \Qu^{pq}(R)$ consists of $J$-Galois algebras which are free as $R$-modules.
\end{Pro} 
 
\begin{proof} We will prove only part iii) as the rest is straightforward to check. Take a $J$-Galois algebra $A$. Thanks to \cite[Corollary 2.5]{kr}, the submodule spanned by $1$ is a free $R$-module, which is a direct summand of $M$ as an $R$-module. Hence, $A=R\cdot 1\oplus V$, where $N$ is projective of rank $1$. If $A$ is free as an $R$-module, then $V$ is a  free $R$-module of rank $1$  (see \cite[Lemma 1.1]{kr}. Denote a generator of $V$ by $v$.  It follows that $v^2=mv +b$, $\eta(v)=l_0\cdot 1\otimes 1+ l_1 \cdot v\otimes 1 +a \cdot 1\otimes x +l_3 \cdot v\otimes x$ and $\eta(1)=1\otimes 1$, for some $a,b,m,l_0.l_1,l_3\in R.$ Since $(\id \otimes \epsilon)\circ \eta(v)=v$, where $\epsilon: A\to R$ is the counit of $J$ and $\epsilon(1)=1, \epsilon (v)=0$, one obtains  $l_0=0$ and $l_1=1$. Thus
\begin{equation}\label{l} 
\eta(v)=v\otimes 1+a \cdot 1\otimes x+l_3 \cdot v\otimes x.
\end{equation}
The condition $(id\otimes \Delta)\circ \eta(v)=(\eta \otimes id)\circ \eta(v)$, implies that
\begin{equation}\label{l3}  al_3=ap \ \ \ {\rm and} \ \ l_3^2=l_3p.
\end{equation}
Next, consider $1\otimes 1, 1\otimes v, v\otimes 1, v\otimes v$ as a base of $A\otimes A$, and $1\otimes 1, v\otimes 1, 1\otimes x, v\otimes x$ as a base of $A\otimes J$. One easily computes that the matrix of $(\mu \otimes id)\circ (id \otimes \eta):A\otimes A\to A\otimes J$ in the chosen bases is 
$$\begin{pmatrix} 1 & 0& 0& b\\ 0& 1& 1&m\\ 0&0& a &l_3b\\ 0 &0& l_3& a+l_3m 
\end{pmatrix}$$
Hence, $(\mu \otimes id)\circ (id \otimes \eta):A\otimes A\to A\otimes J$ is an isomorphism if and only if \begin{equation}\label{sheb}a^2+al_3m-l_3^2b\in R^*.
\end{equation}
It follows from the equations (\ref{l3}) that
$$(l_3-p)(a^2+al_3m-l_3^2b)=(l_3-p)a\cdot a+(l_3-p)a\cdot l_3m-(l_3-p)l_3\cdot l_3b=0$$ 
and by the condition (\ref{sheb}) that $l_3=p$. Thus, we have
\begin{equation}\label{lp} 
\eta(v)=v\otimes 1+a \cdot 1\otimes x+p \cdot v\otimes x \ \  \ \  {\rm and} \ \ \ a^2+pam-p^2b\in R^*.
\end{equation}
As, $\eta$ is an algebra homomorphism, we have
$$m\cdot v\otimes 1+ma\cdot 1\otimes x+ mp\cdot v\otimes x +b \cdot 1\otimes 1=\eta(mv+b)=\eta(v^2)=(v\otimes 1+a\cdot 1\otimes x+ p \cdot v\otimes x)^2,$$
from which we obtain that
$$ am=a^2q, \ \ \ {\rm and} \ \ \   mp=-2a.$$
Lastly, we have
$$(m-aq)(a^2+pam-p^2b)=0,$$
and thus $m=aq$. It follows that $v^2=aqv+b$ and $ a^2+pam-p^2b=-(a^2+p^2b)\in R^*$. Hence, $A$ is in the image of the functor 
$\Qu^{pq}_f(R)\to \Qu^{pq}(R)$s. 
\end{proof}

\begin{Co} The stackification of $R\mapsto \Qu^{pq}_f(R)$ is $R\mapsto \Qu^{pq}(R)$.
\end{Co}

\begin{proof} 
The descent theory of projective modules implies that the assignment $R\mapsto \Qu^{pq}(R)$ is a stack, which, by abuse of notation, we denote by $\Qu^{pq}$. Proposition \ref{43} shows that the natural morphism $\Qu_f\to \Qu$ is a weak equivalence in the sense of \cite[Definition 2.3]{im} and hence the result follows.
\end{proof}

\subsection{The stack $\Pic$}

Let $R$ be a commutative ring. We let $\Pic(R)$ be the following cat-group: The Objects of $\Pic(R)$ are projective $R$-modules of rank $1$, while its morphisms are isomorphisms of $R$-modules. The tensor product $\otimes_R$ equips $\Pic(R)$ with the structure of a cat-group. If $f:R\to S$ is a homomorphism of commutative rings, the functor $f_*:\Pic(R)\to \Pic(S)$, given by $f_*(M)=M\otimes_RS$, is a morphism of cat-groups. 

It follows from the descent theory that $R\mapsto \Pic(R)$ is a stack, which we denote by $\Pic$. Actually, it is the stackification of the prestack assigning to $R$ the cat-group corresponding to the homomorphism $R^*\to \{1\}$ (see \cite[Example (ii), p. 11]{im}). 

It is well-known that 
$$\pi_1(\Pic(R))=R^*, \ \ \ {\rm and} \ \ \pi_0(\Pic(R))={\sf Pic}(R).$$ 
 
\subsection{The stack $\Dis$} 

Objects of the category $\Dis(R)$ are pairs $(M,\mu)$, where $M$ is a finitely generated projective $R$-module of rank $1$ and $\mu$ is a nonsingular, symmetric, bilinear form on $M$. Morphisms $(M,\mu)\to (M',\mu')$ are isomorphisms of $R$-modules $h:M\to M'$, such that $\mu'(h(x),h(y))=\mu(x,y)$. Define the symmetric monoidal structure on $\Dis(R)$ by $(M',\mu')\otimes (M'',\mu'')=(M\otimes M,\mu)$, where 
$$\mu(m_1'\otimes m_1'',m_2'\otimes m_2'')=\mu'(m_1',m_2')\mu''(m_1'',m_2'').$$
The neutral object is $(R,\mu)$, where $\mu(r_1, r_2)=r_1r_2.$

\begin{Le} One has an equivalence of cat-groups:
$$\Dis(R)\cong 2\text{-}\ker(sq_R:\Pic(R)\to \Pic(R)).$$
In particular, $\Dis$ is a stack, which is the stack associated to the prestack $R\mapsto \G(R)$.
\end{Le}

\begin{proof} Let $\mu$ be a nonsingular, bilinear form on a projective module $M$ of rank $1$. Then $\mu:M\otimes_RM\to R$ is an isomorphism, where $R$ is seen as the trivial module, and hence defines an object of $2\text{-}\ker(sq_R:\Pic(R)\to \Pic(R))$.
One easily checks that this construction yields an equivalence of categories 
$$\Dis(R)\cong 2\text{-}\ker(sq_R:\Pic(R)\to \Pic(R)).$$ 
We use Corollary \ref{liftedker} to show the last statement. It says that the cat-group $\G(R)$ is the 2-kernel of the morphism of cat-groups (associated to the vertical arrows), corresponding to the commutative diagram of abelian groups
$$\xymatrix{R^*\ar[r]^{sq}\ar[d]&R^*\ar[d] \ \\ 1\ar[r] &1.}$$
We see that the stackification of $R\mapsto \G(R)$ is the 2-kernel of $sq:\Pic\to \Pic$, which is $\Dis$, since stackification preserves 2-kernels.
\end{proof}

Note that equivalent prestacks give rise to equivalent stackifications and recall the exact sequence (\ref{ex_gz}) on page \pageref{ex_gz}. Denote $\dis(R):=\pi_0(\Dis(R))$. We have the following results:
 
\begin{Co} i) There is an exact sequence (\cite[(12.4) on p. 176]{hahn})
$$0\to \sU_2(R)\to \dis(R)\to \, _2{\sf Pic}(R)\to 0.$$

ii) We have an isomorphism of abelian groups
$$\pi_1(\Dis(R))\cong \mu_2(R)=\{x\in R| x^2=1\}.$$

iii) Propositions \ref{eqcat} and \ref{43} imply that one has an equivalence of cat-groups 
$$t_*^+:\Qu^{pq}(R)\to \Qu^{pt,t^{-1}q}(R).$$

iv) From Corollary \ref{44} iii) we immediately get that $\Qu^{1,-2}\cong \Dis$. If $2$ is invertible in $R$, the second of the remark on page \pageref{remii} yields the equivalence of cat-groups
$$\Qu^{-2,1}(R)\cong \Dis(R).$$ 
After applying $\pi_0$, we get the classical result (\cite[(12.8) on p. 183]{hahn})
$$\sQu(R)\cong {\sf Dis}(R).$$
\end{Co}

\subsection{The stack $\Dis^{pq}$} 

We will need some morphism of stacks. The first one $sq:\Pic\to \Pic$ is given by $sq(M)=M\otimes_RM$, $M\in \Pic(R).$

\begin{Pro} There is a morphism of cat-groups
$$Sq:\Pic(R/pR)\to \Pic(R/p^2R),$$
for which the following diagram
$$\xymatrix{\Pic(R/p^2R)\ar[r]\ar[d]_{sq}&\Pic(R/pR)\ar[d]^{sq}\ar[dl]_{Sq}\\
\Pic(R/p^2R)\ar[r] &\Pic(R/pR)}$$
commutes up to natural isomorphisms. Here, the horizontal arrows are induced by the obvious ring homomorphism $R/p^2R\to R/pR$.
\end{Pro}

\begin{proof} Since the kernel of the ring homomorphism $R/p^2R\to R/p$ is a nilpotent ideal, the functor $\Pic(R/p^2R)\to \Pic(R/pR)$ is full, essentially surjective  and reflects isomorphisms (see for example \cite[Proposition 2.12]{bass}). Hence, the result follows from Lemma \ref{kv} below.
\end{proof}

\begin{Le}\label{kv}
Let $f,g:M\to N$ be two morphisms in $\Pic(R/p^2R)$ such that the induced maps $f_*,g_*:M/pM\to N/pN$ coincide. Then, $sq(f)=sq(g):sq(M)\to sq(N)$.
\end{Le}

\begin{proof} This result can be checked locally, so we may assume that $M=N=R/p^2R$. There are $a, b\in R/p^2R$ such that $f(x)=ax$ and $g(x)=bx$ for all $x\in R/p^2R$. By assumption, $b=a+pr$ for some $r\in R/p^2R$. We have
$$sq(g)(xy)=g(x)g(y)=b^2xy=(a+pr)^2xy=a^2xy-p^2qar=a^2xy=f(x)f(y)=sq(f)(xy).$$
\end{proof} 
  
We now put
$$\Dis^{pq}(R)= 2\text{-}\ker(Sq:\Pic(R/pR)\to \Pic(R/p^2R)).$$
Since the 2-kernel of a morphism of stacks is again a stack, we see that $\Dis^{pq}$ is the stack associated to the prestack $R\mapsto \G^{pq}(R)$.

\begin{Le} 
One has natural isomorphisms 
$$\pi_0(\Dis^{pq}(R))={\sf Dis}(R/p^2R) \ \ {\rm and} \ \ \pi_1(\Dis^{pq}(R))\cong \ker(Sq:(R/pR)^*\to (R/p^2R)^*).$$
\end{Le}

\begin{proof} By our construction, we have a diagram of cat-groups
$$\xymatrix{\Pic(R/p^2R)\ar[r]^{sq}\ar[d]&\Pic(R/p^2R)\ar[d]^{id} \ \\ \Pic(R/pR)\ar[r]_{Sq}&\Pic(R/p^2R)}$$ 
which commutes up to a natural isomorphism. We apply the exact sequence (\ref{ex_gz}), p.\pageref{ex_gz} to the horizontal morphisms. It follows from the discussion on \cite[p. 88] {gz} that one obtains a commutative  diagram with exact rows
$$\xymatrix@C-11pt{0\ar[r]&\mu_2(R/p^2R) \ar[r] \ar[d] &(R/p^2R)^*\ar[r]^{sq}\ar[d]& (R/p^2R)^*\ar[r]\ar[d]^{id}& {\sf Dis}(R/p^2R)\ar[r]\ar[d]& {\sf Pic}(R/p^2R)\ar[r]^{sq}\ar[d]&{\sf Pic}(R/p^2R)\ar[d]^{id}\ \\
0\ar[r]&\pi_1(\Dis^{pq}(R)) \ar[r]  &(R/pR)^*\ar[r]_{Sq}& (R/p^2R)^*\ar[r]& \pi_0(\Dis^{pq}(R))\ar[r]& {\sf Pic}(R/pR)\ar[r]_{Sq}&{\sf Pic}(R/p^2R).}$$ 

This already imply the result for $\pi_1(\Dis^{pq}(R))$. Since the kernel of $R/p^2R\to R/pR$ is a nilpotent ideal, the vertical map
${\sf Pic}(R/p^2R)\to {\sf Pic}(R/pR)$ is an isomorphism and $(R/p^2R)^*\to (R/pR)^*$ is an epimorphism (\cite[Proposition IX.3.4] {bass}). Hence the result is also true for $\pi_0(\Dis^{pq}(R))$, thanks to the 5-lemma.
\end{proof} 

\subsection{The main result} 

After stackification of the prestacks  involved in the diagram described in Theorem \ref{exact_{fpq}}, one obtains the diagram of stacks
\begin{center} $\xymatrix{ \Qu\ar[r]_{\alpha}\rruppertwocell<12>^{1}{^\delta}
&\Dis\ar[r]_{\beta}&\Dis^{pq}.}$
\end{center}
By general properties of stackification, Theorem \ref{exact_{fpq}} i) implies that the the induced functor \begin{center} $\gamma:\Qu\to 2\text{-}\ker(\beta)$
\end{center}
is locally essentially surjective.

Assume that $p$ is not a zero-divisor in $R$. One easily sees that the same property holds for the localization ring $R_r$, where  $r\in R$. It follows that $p$ is not a zero divisor for any of the rings involved in the process of stackification. By Theorem \ref{exact_{fpq}} ii), the induced morphism
\begin{center} $\gamma(R):\Qu^{pq}(R)\to 2\text{-}\ker(\Dis(R)\xto{\beta(R)} \Dis^{pq}(R))$
\end{center}
is an equivalence of cat-groups. The following is now a specialization of the exact sequence (\ref{ex_gz}):

\begin{Th}\label{exactpq} If $p$ is not a zero divisor of $R$, one has an exact sequence of abelian groups 
$$0\to \mathbb{Z}_{pq} (R)\to \mu_2(R)\to {\sf Im}(\mu_2(R/p^2R)\to \mu_2(R/pR))\to \sQu^{pq}(R)\to {\sf Dis}(R)\to {\sf Dis}(R/p^2R).$$
\end{Th}

\end{document}